\documentclass[12pt,twoside]{article}
\usepackage{amsmath, amsthm, amscd, amsfonts, amssymb, graphicx, color}

\begin{document}

\centerline{}

\centerline{}

\centerline{\Large{\bf Tight $J$-frames in Krein space}}

\centerline{}

\centerline{\Large{\bf and the associated $J$-frame potential}}

\centerline{}

\centerline{\bf {Sk. Monowar Hossein}}

\centerline{Department of Mathematics }

\centerline{Aliah University, IIA/27 New Town, Kolkata-156}

\centerline{West Bengal, India.}

\centerline{\textit{email} : sami$\_ $milu@yahoo.co.uk}

\centerline{}

\centerline{\bf {Shibashis Karmakar}}

\centerline{Department of Mathematics }

\centerline{Jadavpur University, Jadavpur-32}

\centerline{West Bengal, India.}

\centerline{\textit{email} : shibashiskarmakar@gmail.com}

\centerline{}

\centerline{\bf {Kallol Paul}}

\centerline{Department of Mathematics }

\centerline{Jadavpur University, Jadavpur-32}

\centerline{West Bengal, India.}

\centerline{\textit{email} : kalloldada@gmail.com}

\newtheorem{theorem}{Theorem}[section]
\newtheorem{lemma}[theorem]{Lemma}
\newtheorem{proposition}[theorem]{Proposition}
\newtheorem{corollary}[theorem]{Corollary}
\theoremstyle{definition}
\newtheorem{definition}[theorem]{Definition}
\newtheorem{example}[theorem]{Example}
\newtheorem{exercise}[theorem]{Exercise}
\newtheorem{conclusion}[theorem]{Conclusion}
\newtheorem{conjecture}[theorem]{Conjecture}
\newtheorem{criterion}[theorem]{Criterion}
\newtheorem{summary}[theorem]{Summary}
\newtheorem{axiom}[theorem]{Axiom}
\newtheorem{problem}[theorem]{Problem}
\theoremstyle{remark}
\newtheorem{remark}[theorem]{Remark}
\numberwithin{equation}{section}
\centerline{}


\begin{abstract}
Motivated by the idea of $J$-frame for a Krein space $\textbf{\textit{K}}$, introduced by Giribet \textit{et al.} (J. I. Giribet, A. Maestripieri, F. Martínez Per\'{i}a,  P. G. Massey, \textit{On frames for Krein spaces}, J. Math. Anal. Appl. (1), {\bf 393} (2012),
122--137.), we introduce the notion of $\zeta-J$-tight frame for a Krein space $\textbf{\textit{K}}$. In this paper we characterize $J$-orthonormal basis for $\textbf{\textit{K}}$ in terms of $\zeta-J$-Parseval frame. We show that a Krein space is richly supplied with $\zeta-J$-Parseval frames. We also provide a necessary and sufficient condition when the linear sum of two $\zeta-J$-Parseval frames is again a $\zeta-J$-Parseval frame. We then generalize the notion of $J$-frame potential in Krein space from Hilbert space frame theory. Finally we provided a necessary and sufficient condition for a $J$-frame potential of the corresponding $\zeta-J$-tight frame to be minimum.
\end{abstract}

{\bf Mathematics Subject Classification:} 42C15, 46C05, 46C20

{\bf Keywords:} Krein Space, Anti-Hilbert space, Grammian operator, $J$-frame, Frame potential.

\section{Introduction}
Now a days Hilbert space frame theory is a very well known concept. The foundation of the theory was initiated by Duffin and Schaeffer (\cite{ds}) in the year 1952,
when they studied a class of nonharmonic Fourier series. But the building blocks were established by the landmark paper of Daubechies \textit{et al.} (\cite{dgm}).
Since then this subject attracted many mathematicians. In finite dimensional Hilbert space the characterization of finite normalized tight frames \textit{i.e.} FNTFs leads directly to
many fascinating problems in application. To study the distribution of FNTFs in space Benedetto and Fickus (\cite{jm}) developed the notion of frame potential which is analogous to potential energy of a system in Physics.

The generalization of the ``Hilbert space" results to the case of indefinite metrics is an interesting and important
problem. Therefore, it is a natural demand to introduce the concept of frames for such type of spaces like Krein space.  Krein spaces play an important role in modern analysis because of it's rich underlying theory and applications in Special Relativity, High energy physics,
Quantum cosmology, Krein space filtering and many more areas. Some works already had been done in this direction. Giribet \textit{et al.} (\cite{gmmm}) defined frames for Krein spaces which they
called as $J$-frames. It is an extension of $J$-orthonormal basis. Similarly Esmeral \textit{et al.} (\cite{koe}) considered a more direct approach. They have shown that frames for a Krein space
and the theory of frames for the associated Hilbert space are equivalent. Recently Esmeral \textit{et al.} (\cite{paef}) studied frames of subspaces of Krein space as well.
Since the definition of Giribet \textit{et al.} (\cite{gmmm}) is more geometric in nature, therefore we consider their definition for further progress in this direction.
 As for practical use we need tight/Parseval frames for finite dimensional spaces, so we are motivated to do so.
In this paper we introduce the notion of $J$-tight frame for a Krein space $\textbf{\textit{K}}$, which is a natural extension of $J$-frame for the Krein space $\textbf{\textit{K}}$.
Then we move our attention to $J$-frame potential of the Krein space $\textbf{\textit{K}}$. At first we define $J$-frame force between a positive and a negative frame vector, then develop the associated frame potential of a given $J$-frame. Finally we give a necessary and sufficient condition when the value of the $J$-frame potential corresponding to a given $J$-frame is minimum.

\section{Preliminary Notes}

\subsection{On Krein spaces}
We briefly mention the definitions, geometric interpretations and some basic properties of Krein spaces focusing on those results that we need for our study \cite{jb,isty}.
\begin{definition}
An abstract vector space $(\textbf{\textit{K}},[~,~])$ that satisfies the following requirements is called a Krein space.\\

 (i) $\textbf{\textit{K}}$ is a linear space over the field $F$, where $F$ is either $\mathbb{R}$ or $\mathbb{C}$.\\

  (ii) there exists a bilinear form $[~,~]\in{F}$ on \textbf{\textit{K}} such that\\
  $$[y,x]=\overline{[x,y]}$$
  $$[ax+by,z]=a[x,z]+b[y,z]$$
  for any $x,y,z\in{\textbf{\textit{K}}}$, $a,b\in{F}$, where $\overline{[~,~]}$ denote the complex conjugation.\\
	
  (iii) The vector space $\textbf{\textit{K}}$ admits a canonical decomposition $\textbf{\textit{K}}=\textbf{\textit{K}}^+[\dot{+}]\textbf{\textit{K}}^-$ such that $(\textbf{\textit{K}}^+,[~,~])$ and $(\textbf{\textit{K}}^-,-[~,~])$ are Hilbert spaces relative to the norms $\|x\|=[x,x]^{\frac{1}{2}}(x\in{\textbf{\textit{K}}^+})$ and $\|x\|=(-[x,x]^{\frac{1}{2}})(x\in{\textbf{\textit{K}}^-})$.
\end{definition}
Now every canonical decomposition of $\textbf{\textit{K}}$ generates two mutually complementary projectors $P_+$ and $P_-$ ($P_{+}+P_-=I$, the identity operator on $\textbf{\textit{K}}$ ) mapping $\textbf{\textit{K}}$ onto $\textbf{\textit{K}}^+$ and $\textbf{\textit{K}}^-$ respectively. Thus for any $x\in{\textbf{\textit{K}}}$, we have $P_{\pm}=x^{\pm}$, where $x^+\in{\textbf{\textit{K}}^+}$ and $x^-\in{\textbf{\textit{K}}^-}$. The projectors $P_+$ and $P_-$ are called canonical projectors.

The linear operator $J:\textbf{\textit{K}}\to{\textbf{\textit{K}}}$ defined by the formula $J=P_+-P_-$ is called the canonical symmetry of the Krein space $\textbf{\textit{K}}$. The canonical symmetry $J$ immediately generates orthogonal canonical projectors $P_{\pm}$ according to the formula $P_{\pm}=\frac{1}{2}(I\pm{J})$ and a canonical decomposition $\textbf{\textit{K}}=\textbf{\textit{K}}^+\oplus\textbf{\textit{K}}^-,~\textbf{\textit{K}}^{\pm}=P_{\pm}\textbf{\textit{K}}$ and also the $J$-metric defined by the formula $[x,y]_{J}=[x,Jy]$, where $x,y\in{\textbf{\textit{K}}}$. The vector space $\textbf{\textit{K}}$ associated with the $J$-metric is a Hilbert space, called the associated Hilbert space of the Krein space $\textbf{\textit{K}}$.

\subsection{Basics of Hilbert Space Frame Theory }
A family of vectors $\{f_i\}_{i\in I}$ is said to be a frame for a Hilbert space $H$, if there exists positive real numbers $A$ and $B$ with $A \leq B$ such that
\begin{equation}
 A \| f \|^2 \leq \sum_{i \in I} |\langle f,f_i \rangle|^2  \leq B \| f \|^2
\end{equation}
for all $f \in H$. $A,B$ are known as lower and upper frame bounds respectively for the frame. If $A = B$ then the frame is known as $A$-tight frame and if $A=B=1$ then the frame
 is known as Parseval frame.\\
 Let $\{f_i\}_{i\in I}$ be a frame for the Hilbert space $H$ and $\{e_i\}_{i\in I}$ be the natural orthonormal basis of $\ell_2(I)$. An operator $ T:H \to \ell_2 (I)$ defined by
 $T(f)=\sum_{i \in I} \langle f,f_i \rangle e_i$ for all $f\in H$ is  known as analysis operator and its adjoint operator defined by
 $ T^* (e_i)=f_i $ is known as synthesis operator for the frame $\{f_i\}_{i\in I}$. The operator $ S (=TT^*):H \to H $ given by
 $ S(f)=\sum_{i \in I} \langle f,f_i \rangle f_i$ for all $f\in H$ is called frame operator. A mapping $G:\ell_2(I)\to\ell_2(I)$ defined as $G = T^* T$ is known as the Grammian operator.
 It is clear that $S$ is self-adjoint, positive and invertible operator and
 $A.I \leq S \leq B.I$

A normalized tight frame (NTF) is a tight frame $\{f_i\}_{i \in I}$ with $\|f_i\| = 1$ for all $ i \in I$. A finite NTF, denoted by FNTF possess a significant structure. We will mention some of the results on FNTF of $N$ elements for a $d$-dimensional Hilbert space $H = F^d $.
\begin{theorem} \cite{jm}
If $\{x_n\}^N _{n=1}$ is an A-FNTF for a $d$-dimensional Hilbert space H, then $ A = \frac{N}{d} $.
\end{theorem}
\begin{theorem}\cite{gkk}
 Given any $d$ and $ N \geq d$, then there exists a FNTF for $F^d$ of $N$ elements.
\end{theorem}
\begin{theorem}\cite{gkk}
 Normalized tight frames for $\mathbb{R}^2$ of $N$ elements correspond to sequences $\{z_n\}^N _{n=1} \subseteq \mathbb{C}$ with $|z_n| = 1$ for all n and for which $\sum_{n=1} ^N z_n ^2 = 0.$
\end{theorem}
The above theorem state that for $N>2$ the $N$ vectors which are $N$th roots of unity form a FNTF for $\mathbb{R}^2$. So it provides a technique to find out FNTFs for $\mathbb{R}^2$.
Let us consider $d$-dimensional real Hilbert space $\mathbb{R}^d$ and $S^{d-1}$ be the unit sphere in $\mathbb{R}^d$. Then frame force \cite{jm} between any two ponts on $S^{d-1}$ is a function
\begin{equation}
 FF: S^{d-1} \times S^{d-1} \to \mathbb{R}^d, \mbox{ defined by } FF(a,b) = \langle a,b \rangle (a-b).
\end{equation}
As we know that a central force $FF$  corresponds to a real valued continuous function $f$ on $ (0,2] $ such that $F(a,b)=f(\|a-b\|)(a-b)~\textmd{for all}~a,b\in S^{d-1}$, we will always have a function $f$ such that $ FF(a,b)=f(\|a-b\|)(a-b)$. \\
Let $D= \{(x,x) : x\in S^{d-1}\}$, then for a given central force $FF$, the potential \cite{jm} corresponding to $FF$ is a function $P:(S^{d-1} \times S^{d-1})\backslash{D} \to \mathbb{R}$, defined by $P(a,b) = p(\|a-b\|)$, where $p:(0,2] \to R $ satisfies $p'(x) = -xf(x)$.
\begin{definition}
The frame potential in $F^d$ is a function $FP: (S^{d-1})^N \to [0,\infty)$, defined
as $ FP( \{y_n\}_{n=1} ^N) = \sum_{i=1} ^ N \sum_{j=1} ^ N |\langle y_i,y_j \rangle|^2 $, where $S^{d-1} = \{ x\in F^d : \|x\|=1\}$.
\end{definition}
Note that difference of frame potential is the work done to transform one sequence into other.
\begin{theorem}\cite{jm}
 Let $\{y_n\}_{n=1} ^N $ be a sequence in $F^d$ with associated frame operator $S$. Then $ FP(\{y_n\}_{n=1} ^N) = Tr(S^2).$
\end{theorem}
The main results on frame potential for Hilbert space are
\begin{theorem}\cite{jm}
 If $N \leq d$, the minimum value of the frame potential is $N$, and the minimizers are precisely the orthonormal  sequences in $F^d$.
\end{theorem}

\begin{theorem}\cite{jm}
 If $N \geq d$, the minimum value of the frame potential is $\frac{N^2}{d}$, and the minimizers are precisely the FNTFs for $F^d$.
\end{theorem}

\section{Main Results}
\subsection{Definition of $J$-frames in Krein Spaces}
Let $(\textbf{\textit{K}},[~.~],J)$ be a Krein space. Suppose $\textbf{\textit{F}}=\{f_n:n\in{\mathbb{N}}\}$ is a Bessel sequence of $\textbf{\textit{K}}$ and $T\in{L(\ell^2(I),\textbf{\textit{K}})}~(\ell^2(I):=\{(c_i):\sum_{i\in{I}}|c_i|^2<\infty\})$ is the synthesis operator for the Bessel sequence $\textbf{\textit{F}}$. Let $I_+=\{i\in{I}:[f_i,f_i]\geq{0}\}$ and $I_-=\{i\in{I}:[f_i,f_i]<0\}$, then $\ell^2(I)=\ell^2(I_+)\bigoplus{\ell^2(I_-)}$. Also let $P_{\pm}$ denote the orthogonal projection of $\ell^2(I)$ onto $\ell^2(I_{\pm})$. Let $T_{\pm}=TP_{\pm}$, $M_{\pm}=\overline{span}\{ f_i:i\in{I_{\pm}}\}$ then we have $R(T)=R(T_+)+R(T_-)$, where $R(T)$ represents range of the operator $T$.
\begin{definition}\cite{gmmm}
 A Bessel sequence $\textbf{\textit{F}}$ is said to be a $J$-frame for $\textbf{\textit{K}}$ if $R(T_+)$ is a maximal uniformly $J$-positive subspace of $\textbf{\textit{K}}$ and $R(T_-)$ is a maximal uniformly $J$-negative subspace of $\textbf{\textit{K}}$.
\end{definition}

Now every $J$-frame is associated with a positive real numbers $\zeta$, where $\zeta=c_0(M_+,\verb"C")+c_0(M_-,\verb"C")$. Let $G_M$ be the Grammian operator of $M$ \cite{gmmm} and $ \gamma(T)$ is the reduced minimum modulus of an operator $T$ \cite{gmmm}. Then, $c_0(M_+,\verb"C")=\frac{1}{\sqrt2}(\sqrt{\frac{1+\alpha^+}{2}}+\sqrt{\frac{1-\alpha^+}{2}})$ and
$c_0(M_-,\verb"C")=\frac{1}{\sqrt2}(\sqrt{\frac{1+\beta^+}{2}}+\sqrt{\frac{1-\beta^+}{2}})$, where $\alpha^+=\gamma({G_{M_+}})$ and $\beta^+=\gamma({G_{M_-}})$ and $\verb"C"=\{n\in{\textbf{\textit{K}}}:[n,n]=0\}$. \\
We also have $\zeta~\in[\sqrt{2},2)$. We will use the real number $\zeta$ (associated with a $J$-frame for $\textbf{\textit{K}}$) extensively in our work and instead of the term $J$-frame for $\textbf{\textit{K}}$ we will use the term $\zeta-J$-frame for $\textbf{\textit{K}}$.
\begin{definition}
Let $(\textbf{\textit{K}},[~.~],J)$ be a Krein space and $\textbf{\textit{F}}=\{f_n:n\in{\mathbb{N}}\}$ be a $J$-frame for the Krein space $\textbf{\textit{K}}$. Then $\textbf{\textit{F}}$ is said to be a $\zeta-J$-tight frame iff
\begin{equation}
A_{\pm}[f,f]=\sum_{i\in{I_{\pm}}}|[f,f_i]|^2,~~~~\textmd{ for all}{~f\in{M_{\pm}}}
\end{equation}
Moreover, $\textbf{\textit{F}}$ is said to be a $\zeta-J$-Parseval frame if it is a $\zeta-J$-tight frame for the Krein space $\textbf{\textit{K}}$ and in addition $A_{\pm}={\pm}1$.
\end{definition}
\begin{definition}
Let $\textbf{\textit{F}}=\{f_n:n\in{I}\}$ is a $\zeta-J$-frame for the Krein space $\textbf{\textit{K}}$. Then $\textbf{\textit{F}}$ is said to be a normalized $\zeta-J$-frame if $\|f_i\|_J=1$.
\end{definition}
\begin{definition}
Let $S_\textbf{\textit{K}}=\{x\in{\textbf{\textit{K}}}:\|x\|_J=1\}$, $S_{M_+}=\{x\in{M_+}:[x,x]=1\}$ and $S_{M_-}=\{x\in{M_-}:[x,x]=-1\}$. Then a $\zeta-J$-frame $\textbf{\textit{F}}$ is said to be weakly normalized if $\{f_i:i\in{I_+}\}\subset{S_{M_+}}$ and $\{f_i:i\in{I_-}\}\subset{S_{M_-}}$.
\end{definition}
\begin{example}
Consider the Vector space $\mathbb{R}^3$ over $\mathbb{R}$. Let us define a inner product on $\mathbb{R}^3$ by $[x,y]=x_1 y_1 + x_2 y_2-x_3 y_3 $ when $x=(x_1,x_2,x_3), y=(y_1,y_2,y_3)$ and $ x_i, y_i\in\mathbb{R}$ for $i = 1,2,3$. Consider the vectors $\{\frac{\sqrt2}{\sqrt3}(-\frac{\sqrt3}{2},-\frac{1}{2},0),\\
\frac{\sqrt2}{\sqrt3}(\frac{\sqrt3}{2},-\frac{1}{2},0),\frac{\sqrt2}{\sqrt3}(0,1,0),(\frac{1}{\sqrt2},0,\frac{\sqrt3}{\sqrt2})\}$. It is a $J$-frame for the above Krein space $\mathbb{R}^3$. It is also a $J$-Parseval frame with $\zeta=\frac{3}{2\sqrt2}+\frac{\sqrt3}{2\sqrt2}$. By numerical calculation we have $\alpha^+=1$ and $\beta^+=\frac{1}{2}$.
\end{example}

\subsection{Some results on $J$-tight frames}
 Consider $\sqrt{2}-J$-Parseval frames for a Krein space $\textbf{\textit{K}}$ which is also weakly normalized. Our next theorem will establish a relation between $\sqrt{2}-J$-weakly normalized Parseval frame and $J$-orthonormal basis for the given Krein space.
\begin{theorem}
Let $\textbf{\textit{K}}$ be a Krein space of dimension $N$. Then any $\zeta-J$-frame is a $J$-orthonormal basis for $\textbf{\textit{K}}$ iff it is a $\sqrt{2}-J$-weakly normalized Parseval frame for $\textbf{\textit{K}}$.
\end{theorem}
\begin{proof}
Let $\{e_i:i\in{I}\}$ be a $J$-orthonormalized basis for $\textbf{\textit{K}}$. Then $\{e_i:i\in{I}\}$ spans whole of $\textbf{\textit{K}}$. Therefore, the number of vectors in the index set $I$ is $N$. Since $[e_i,e_j]=\pm\delta_{ij},~\textmd{for all}~{i,j}\in{I}$, so let $I_+=\{i:[e_i,e_i]=1\}$ and $I_-=\{i:[e_i,e_i]=-1\}$. Consider $M_+=\overline{span}\{e_i:i\in{I_+}\}$ and $M_-=\overline{span}\{e_i:i\in{I_-}\}$. Now we have a fundamental decomposition of $\textbf{\textit{K}}$ \textit{i.e.} $\textbf{\textit{K}}=M_+{[\dot{\oplus}]}M_-$ ( The symbol $[\dot{\oplus}]$ used here is for ``two-fold orthogonality", the usual one and also in the $J$-metric ). As $J(M_+)=M_+$ and $J(M_-)=-M_-$, therefore $\gamma{(G_{M_+})}=1$ and also $\gamma{(G_{M_-})}=1$. So we have $c_0(M_+,\verb"C")=\frac{1}{\sqrt{2}}$ and $c_0(M_-,\verb"C")=\frac{1}{\sqrt{2}}$. Hence $\zeta=\sqrt2$. Also $\{e_i:i\in{I}\}$ is both normalized and weakly normalized. Now since $\{e_i:i\in{I_+}\}$ is a orthonormal basis for $(M_+,[~.~])$ and $\{e_i:i\in{I_-}\}$ is a orthonormal basis for $(M_-,-[~.~])$. So $\{e_i:i\in{I}\}$ is a $\sqrt{2}-J$-Parseval frame for the Krein space $\textbf{\textit{K}}$.
\\

 Conversely, let $\{f_i:i\in{I}\}$ be a finite collection of vectors in $\textbf{\textit{K}}$ which is a weakly normalized $\sqrt{2}-J$-Parseval frame. Let $I_+=\{i:[f_i,f_i]=1\}$ and $I_-=\{i:[f_i,f_i]=-1\}$. Consider $M_+=\overline{span}\{f_i:i\in{I_+}\}$ and $M_-=\overline{span}\{f_i:i\in{I_-}\}$. Now from Hilbert space frame theory it is clear that $\{f_i:i\in{I_+}\}$ is a orthonormal basis for $(M_+,[~.~])$ and $\{f_i:i\in{I_-}\}$ is a orthonormal basis for $(M_-,-[~.~])$.\\
 Again we have $c_0(M_+,\verb"C")+c_0(M_-,\verb"C")=\sqrt{2}$, which implies that $c_0(M_+,\verb"C")=\frac{1}{\sqrt{2}}$ and $c_0(M_-,\verb"C")=\frac{1}{\sqrt{2}}$. Also by some simple numerical calculation we have $\gamma{(G_{M_+})}=1$ and also $\gamma{(G_{M_-})}=1$. Using all the above results we conclude that
$\textbf{\textit{K}}=M_+{[\dot{\oplus}]}M_-$.\\
Hence the proof.
\end{proof}
Let $\textbf{\textit{F}}=\{f_i:i\in{I}\}$ be a sequence in a Krein space $(\textbf{\textit{K}},[~.~],J)$. Consider $M_+=\overline{span}\{f_i:i\in{I_+}\}$ and $M_-=\overline{span}\{f_i:i\in{I_-}\}$, where $I_+=\{i:[f_i,f_i]>0\}$ and $I_-=\{i:[f_i,f_i]<0\}$. Now if $M_+$ is a maximal uniformly $J$-positive subspace of $\textbf{\textit{K}}$  and $M_-$ is a maximal uniformly $J$-negative subspace of $\textbf{\textit{K}}$. Then $\{f_i:i\in{I_+}\}$ and $\{f_i:i\in{I_-}\}$ will be a frame for $(M_+,[~.~])$ and $(M_-,-[~.~])$ respectively. Let $T_1$ be the synthesis operator for the frame $\{f_i:i\in{I_+}\}$ and $T_2$ be the synthesis operator for the frame $\{f_i:i\in{I_-}\}$. Let $T_1^\ast$ and $T_2^\ast$ are adjoint operators of $T_1$ and $T_2$ respectively.
\begin{definition}
Let $(\textbf{\textit{K}},[~.~],J)$ be a Krein space. A sequence of non-neutral vectors $\textbf{\textit{F}}=\{f_i:i\in{I}\}$ is said to be a disjoint sequence of $\textbf{\textit{K}}$ if \\
 $\overline{span}\{f_i:i\in{I_+}\}~{\cap}~\overline{span}\{f_i:i\in{I_-}\}=\{0\}$.
\end{definition}
\begin{example}
Every $\zeta-J$-frame in a Krein space $\textbf{\textit{K}}$ is a disjoint sequence of $\textbf{\textit{K}}$.
\end{example}
\begin{definition}
Let $\textbf{\textit{F}}=\{f_i:i\in{I}\}$ be a sequence of non-neutral vectors in the Krein space $\textbf{\textit{K}}$. Then $\textbf{\textit{F}}$ is said to be strictly disjoint sequence iff $\overline{span}\{f_i:i\in{I_+}\}~{[\perp]}~\overline{span}\{f_i:i\in{I_-}\}$.
\end{definition}
We will now derive an useful result regarding $\zeta-J$-Parseval frames for a Krein space $\textbf{\textit{K}}$. Our theorem guarantees
that a Krein space is richly supplied with $\zeta-J$-Parseval frames.
\begin{theorem}
Let $(\textbf{\textit{K}},[~.~],J)$ be a Krein space. Assume that $M_1$ and $M_2$ are definite subspaces of $\textbf{\textit{K}}$ respectively. Let $\{x_i\}$ and $\{y_i\}$ are Parseval frames for $M_1$ and $M_2$ respectively.Then $\{x_i\}\cup\{y_i\}$ is a strictly disjoint sequence in $\textbf{\textit{K}}$ only if $\{x_i\}\cup\{y_i\}$ is a $\zeta-J$-Parseval frame for $\textbf{\textit{K}}$ for $\zeta{\in}[\sqrt{2},2)$.
\end{theorem}
\begin{proof}
Without any loss of generality we assume that $M_1$ is positive $J$-definite. Hence it is intrinsically complete. Hence $(M_1,[~.~])$ is a Hilbert space. Given that $\{x_i:i\in{I_1}\}$ is a Parseval frame for $(M_1,[~.~])$. Therefore, $\overline{span}\{x_i:i\in{I_1}\}=M_1$ and also $[x_i,x_i]>0$.

Now since $\{x_i\}\cup\{y_i\}$ is a strictly disjoint sequence in $\textbf{\textit{K}}$, so $\overline{span}\{x_i\}~{[\perp]}~$\\
$\overline{span}\{y_i\}=\{0\}$. Hence $\overline{span}\{y_i\}$ is a closed negative subspace of $\textbf{\textit{K}}$. Now $\{y_i:i\in{I_2}\}$ is a Parseval frame for $M_2$. Therefore, $\overline{span}\{y_i:i\in{I_2}\}=M_2$. Hence $M_2$ is a negative $J$-definite subspace of $\textbf{\textit{K}}$. Since $M_1~[\perp]~M_2$, so we have a fundamental decomposition of $\textbf{\textit{K}}$ \textit{i.e.} $\textbf{\textit{K}}=M_1~[\dot{+}]~M_2$ (see \cite{isty}). Since both $M_1$ and $M_2$ is closed $J$-definite and also intrinsically complete, hence both $M_1$ and $M_2$ are uniformly $J$-definite (see \cite{jb}). Let $\zeta=c_0(M_+,\verb"C")+c_0(M_-,\verb"C")$, then $\{x_i\}\cup\{y_i\}$ is a $\zeta-J$-frame for $\textbf{\textit{K}}$. As both $\{x_i\}$ and $\{y_i\}$ are Parseval frames for $M_1$ and $M_2$ respectively, so $\{x_i\}\cup\{y_i\}$ is also a $\zeta-J$-Parseval frame for $\textbf{\textit{K}}$.\\
\end{proof}
\begin{remark} The statement of the above theorem is sufficient but not necessary. Since let $\{x_i:i\in{I_1}\}\cup\{y_i:i\in{I_2}\}$ is a $\zeta-J$-Parseval frame for $\textbf{\textit{K}}$ for $\zeta{\in}[\sqrt{2},2)$, then there exists uniformly $J$-positive definite subspace $M_+$ and uniformly $J$-negative subspace $M_-$ such that $\{x_i:i\in{I_1}\}$ is a Parseval frame for $(M_+,[~.~])$ and $\{y_i:i\in{I_2}\}$ is a Parseval frame for the Hilbert space $(M_-,[~.~])$. But $M_+$ may not be perpendicular to $M_-$. Hence $\{x_i\}\cup\{y_i\}$ may not be a strictly disjoint sequence in $\textbf{\textit{K}}$.
\end{remark}
The following result for $\zeta-J$-Parseval frames describes how $\zeta-J$-Parseval frames can be combined to form a new $\zeta-J$-Parseval frame under some restrictions.
\begin{theorem}
Let $\{f_i:i\in{I}\}$ and $\{g_i:i\in{I}\}$ are $\zeta-J$-Parseval frames for $(\textbf{\textit{K}},[~.~],J)$ such that $M_+=\overline{span}\{g_i:i\in{I_{+}}\}=\overline{span}\{f_i:i\in{I_+}\}$ and $M_-=\overline{span}\{g_i:i\in{I_{-}}\}=\overline{span}\{f_i:i\in{I_-}\}$. Then for all scalars $\alpha$ and $\beta$ with $|\alpha|^2+|\beta|^2=1$, $\{{\alpha}f_i+{\beta}g_i:i\in{I}\}$ is a $\zeta-J$-Parseval frame for $\textbf{\textit{K}}$ iff $Re(\alpha\overline{\beta}[f_i,g_i])<\frac{{\pm}|\alpha|^2[f_i,f_i]{\pm}\beta|^2[g_i,g_i]}{2},~~\textmd{for all}~{i\in{I_{\pm}}}$,  $T^{\ast}_{1f}T_{1g}+T^{\ast}_{1g}T_{1f}=0$ and $T^{\ast}_{2f}T_{2g}+T^{\ast}_{2g}T_{2f}=0$, where $T_{1f}$ and $T_{1g}$ are synthesis operators of $\{f_i:i\in{I_+}\}$ and $\{g_i:i\in{I_{+}}\}$ respectively and $T_{2f}$ and $T_{2g}$ are synthesis operators of $\{f_i:i\in{I_-}\}$ and $\{g_i:i\in{I_{-}}\}$ respectively.
\end{theorem}
\begin{proof}
Let $\{{\alpha}f_i+{\beta}g_i:i\in{{I}}\}$ be a $\zeta-J$-Parseval frame for $\textbf{\textit{K}}$.\\
So $[{\alpha}f_i+{\beta}g_i,{\alpha}f_i+{\beta}g_i]^{>}_{<}0,~~\textmd{for all}~{i}\in{I}$.

\textit{i.e.} we have $Re(\alpha\overline{\beta}[f_i,g_i])<\frac{|\alpha|^2[f_i,f_i]+|\beta|^2[g_i,g_i]}{2}$

and $Re(\alpha\overline{\beta}[f_i,g_i])<\frac{-|\alpha|^2[f_i,f_i]-|\beta|^2[g_i,g_i]}{2}$.

Since $\{f_i:i\in{I}\}$ is a $\zeta-J$-Parseval frame for $\textbf{\textit{K}}$, therefore $\{f_i:i\in{I_+}\}$ is a Parseval frame for $(M_+,[~.~])$. Let $T_{1f}$ be the synthesis operator of $\{f_i:i\in{I_+}\}$ for the space $(M_+,[~.~])$. Then $T_{1f}:~\ell^2(I_+)\rightarrow~M_+$ is defined by $T_{1f}(\{c_i\})=\sum_{i\in{I_+}}c_if_i$ and $T^{\ast}_{1f}:~M_+\rightarrow~\ell^2(I_+)$ \textit{i.e.} the analysis operator is defined by $T^{\ast}_{1f}(f)=\{[f,f_i]\}_{i\in{I_+}}$. Similarly $\{g_i:i\in{I}\}$ is also a $\zeta-J$-Parseval frame for $\textbf{\textit{K}}$. Hence $\{g_i:i\in{I_{+}}\}$ is a Parseval frame for $(M_+,[~.~])$. So, $T_{1g}$ and $T^{\ast}_{1g}$ are defined as above. Also $\{f_i:i\in{I_-}\}$ and $\{g_i:i\in{I_{-}}\}$ are Parseval frames for $(M_-,-[~.~])$, so we can define the operators \textit{viz.} $T_{2f}$, $T^{\ast}_{2f}$ $T_{2g}$ and $T^{\ast}_{2g}$ as above.

We assume that for all scalars $\alpha$ and $\beta$ such that $|\alpha|^2+|\beta|^2=1$, $\{{\alpha}f_i+{\beta}g_i:i\in{{I}}\}$ is a $\zeta-J$-Parseval frame for $\textbf{\textit{K}}$.
So $\{{\alpha}f_i+{\beta}g_i:i\in{I_+}\}$ is a Parseval frame for $(M_+,[~.~])$. Similarly $\{{\alpha}f_i+{\beta}g_i:i\in{I_-}\}$ is a Parseval frame for $(M_-,-[~.~])$. Let $T_1$ be the synthesis operator for the frame $\{{\alpha}f_i+{\beta}g_i:i\in{{I_+}}\}$ in $(M_+,[~.~])$ and $T_2$ be the synthesis operator for the frame $\{{\alpha}f_i+{\beta}g_i:i\in{{I_-}}\}$ in $(M_-,-[~.~])$. Also let us assume that $T^{\ast}_1$ and $T^{\ast}_2$ be the adjoint operators of $T_1$ and $T_2$ respectively. Now let us define $\mathfrak{T}_1:=~\alpha{T_{1f}}+\beta{T_{1g}}$. Then we have $\mathfrak{T}^{\ast}_1=~\overline{\alpha}T^{\ast}_{1f}+\overline{\beta}T^{\ast}_{1g}$. A direct calculation shows that $\mathfrak{T}_1=T_1$ is the synthesis operator for the frame $\{{\alpha}f_i+{\beta}g_i:i\in{{I_+}}\}$ in $(M_+,[~.~])$.\\
Hence $\mathfrak{T}^{\ast}_1\mathfrak{T}_1=I$. Therefore, $T^{\ast}_{1f}T_{1g}+T^{\ast}_{1g}T_{1f}=0$.\\
Similarly we can show that $T^{\ast}_{2f}T_{2g}+T^{\ast}_{2g}T_{2f}=0$.\\

Conversely let $T{\in}L(\ell^2(I),\textbf{\textit{K}})$ be the synthesis operator for the Bessel family $\{{\alpha}f_i+{\beta}g_i:i\in{{I}}\}$ in the Krein space $(\textbf{\textit{K}},[~.~],J)$.\\
Then the condition $Re(\alpha\overline{\beta}[f_i,g_i])<\frac{{\pm}|\alpha|^2[f_i,f_i]{\pm}\beta|^2[g_i,g_i]}{2},~~\textmd{for all}~{i\in{I_{\pm}}}$ implies that\\
$I_+=~\{i{\in}I:[{\alpha}f_i+{\beta}g_i,{\alpha}f_i+{\beta}g_i]>0\}$ and $I_-=~\{i{\in}I:[{\alpha}f_i+{\beta}g_i,{\alpha}f_i+{\beta}g_i]<0\}$.\\
So we have a orthogonal decomposition of $\ell^2(I)$. Let $P_{\pm}$ denote the orthogonal projection onto $\ell^2(I_{\pm})$, also let $T_{\pm}=~TP_{\pm}$. Now $R(T)=~R(T_+)+R(T_-)$, and $R(T_+)=M_+$ and $R(T_-)=M_-$, since $T(\{c_i\})=\sum_{i{\in}I}c_i({\alpha}f_i+{\beta}g_i)=\alpha\sum_{i{\in}I}c_if_i{+}\beta\sum_{i{\in}I}c_ig_i$. Hence by the definition of $J$-frame, $\{{\alpha}f_i+{\beta}g_i:i\in{{I}}\}$ is a $J$-frame for the Krein space $\textbf{\textit{K}}$.

Now by above we can easily show that the condition $T^{\ast}_{1f}T_{1g}+T^{\ast}_{1g}T_{1f}=0$ implies that $\{{\alpha}f_i+{\beta}g_i:i\in{{I_+}}\}$ is a Parseval frame for $(M_+,[~.~])$ for all scalars $\alpha$ and $\beta$ such that $|\alpha|^2+|\beta|^2=1$, and the condition $T^{\ast}_{2f}T_{2g}+T^{\ast}_{2g}T_{2f}=0$ implies that $\{{\alpha}f_i+{\beta}g_i:i\in{I_-}\}$ is a Parseval frame for $(M_-,-[~.~])$ for all scalars $\alpha$ and $\beta$ such that $|\alpha|^2+|\beta|^2=1$.\\
In order to show that for all scalars $\alpha$ and $\beta$ with $|\alpha|^2+|\beta|^2=1$, $\{{\alpha}f_i+{\beta}g_i:i\in{I}\}$ is a $J$-Parseval frame for $\textbf{\textit{K}}$, we only need to find $\zeta\in{[\sqrt{2},2)}$ such that $\zeta=c_0(M_+,\verb"C")+c_0(M_-,\verb"C")$. Now since the quantities  $c_0(M_+,\verb"C")$ and $c_0(M_-,\verb"C")$ are fixed throughout and $\{f_i:i\in{I}\}$ is a given $\zeta-J$-Parseval frame, so we already have $\zeta$ which satisfies the above equation.\\
Hence the proof.
\end{proof}

\subsection{Frame Potential}
In 2001, John Benedetto and Matthew Fickus \cite{jm} developed a theoretical notion of frame potential analogous to the potential energy \cite{dkde} in Physics. The frame potential of a collection of vectors must be a scalar quantity derived from the inner products between the vectors. In this section we define the frame potential for a finite collection of vectors in Krein space $\textbf{\textit{K}}$, similar to what had been done for frames in Hilbert space theory.

Let $\textbf{\textit{K}}$ be a Krein space of dimension $N$. So if $dim~M_+=m$ and $dim~M_-=n$. Then $m+n=N$. Let $X$ be any finite set and $|X|$:= no. of elements in $X$. Now let $\textbf{\textit{F}}=\{f_i:i\in{I}\}$ be a finite $\zeta-J$-frame for the Krein space $\textbf{\textit{K}}$ with $|\textbf{\textit{F}}|=M$. We have $I_+=\{i:[f_i,f_i]>0\}$ and $I_-=\{i:[f_i,f_i]<0\}$. Considering $|I_+|=p$ and $|I_-|=q$, we have $p+q=M$.

\subsection{$J$-force between two vectors}
Let $\textbf{\textit{F}}=\{f_i\}_{i=1}^M$ be a finite $\zeta-J$-frame for $\textbf{\textit{K}}$,  $M_+$ contain all the positive elements and $M_-$ contain all the negative elements of the frame.\\
Then the $J$-frame force between $f_i,f_j\in{M_+}$ is
\begin{equation}
FF_J(f_i,f_j)=2[f_i,f_j](f_i-f_j),~~~~\textmd{for all}~{i,j\in{I_+}}
\end{equation}
Similarly $J$-frame force between $f_{i^\prime},f_{j^\prime}\in{M_-}$ is
\begin{equation}
FF_J(f_{i^\prime},f_{j^\prime})=-2[f_{i^\prime},f_{j^\prime}](f_{i^\prime}-f_{j^\prime}),~~~~\textmd{for all}~{i^\prime,j^\prime\in{I_-}}
\end{equation}
and $J$-frame force between $f_{i^{\prime\prime}},f_{j^{\prime\prime}}$ where ${i^{\prime\prime}}\in{I_+}$ and ${j^{\prime\prime}}\in{I_-}$ is
\begin{equation}
FF_J(f_{i^{\prime\prime}},f_{j^{\prime\prime}})=2(\|f_{i^{\prime\prime}}\|_J\|f_{j^{\prime\prime}}\|_{J}\zeta+[f_{i^{\prime\prime}},
f_{j^{\prime\prime}}]_J)(f_{i^{\prime\prime}}-f_{j^{\prime\prime}})
\end{equation}
$~~\textmd{for all}~{i^{\prime\prime}\in{I_+},~j^{\prime\prime}\in{I_-}}~\textmd{and~}\zeta~\in[\sqrt{2},2).$

\subsection{$J$-potential between two vectors and total potential}
The $J$-frame force is conservative when the vectors are constricted to lie on sphere of uniformly $J$-positive definite and uniformly $J$-negative definite subspace respectively. So let us consider $f_i,f_j\in{M_+}$ such that $\|f_i\|_+=a_i$ and $\|f_j\|_+=a_j$. Then $\|f_i-f_j\|_+^2=a_i^2-2[f_i,f_j]+a_j^2$, so $FF_J(f_i,f_j)=(a_i^2+a_j^2-\|f_i-f_j\|_+^2)(f_i-f_j),~\textmd{for all}~{i,j\in{I_+}}$. Now $p(x)=-\int(a_i^2+a_j^2-x^2)x~dx=\frac{1}{4}x^2\{x^2-2(a_i^2+a_j^2)\}$, and evaluating at $x=\|f_i-f_j\|_{+}$, the $J$-potential $P_J(f_i,f_j)=p(\|f_i-f_j\|_{+})=[f_i,f_j]^2-\frac{1}{4}(a_i^2+a_j^2)^2,~\textmd{for all}~{i,j\in{I_+}}$.

Similarly the $J$-potential between two negative vectors i.e. $P_J(f_{i^\prime},f_{j^\prime})=[f_{i^\prime},f_{j^\prime}]^2-\frac{1}{4}(a_{i^\prime}^2+a_{j^\prime}^2)^2,~~\textmd{for all}~{{i^\prime},{j^\prime}\in{I_-}}$.
Here $a_{i^\prime}:=\|f_{i^\prime}\|_{-}$ and $a_{j^\prime}:=\|f_{j^\prime}\|_{-}$.

And, finally let $f_{i^{\prime\prime}}\in{M_+}$ and $f_{j^{\prime\prime}}\in{M_-}$. But here the $J$-force between a positive vector and a negative vector is not conservative. So we have a system where there is both conservative and non-conservative forces. There is no definition of potential of a non-conservative force. So we define the $J$-potential between a positive vector and a negative vector in the following way
$$P_J(f_{i^{\prime\prime}},f_{j^{\prime\prime}})=\frac{1}{2}(\zeta^2-1),~~\textmd{for all}~{{i^{\prime\prime}}\in{I_+},{j^{\prime\prime}}\in{I_-}}.$$

Now total $J$-potential of the system \textit{i.e.}\\

$TP_J(\{f_i:i\in{I}\})=\sum_{i,j\in{I_+}}|[f_i,f_j]|^2+\sum_{i,j\in{I_-}}|[f_i,f_j]|^2+\sum_{i\in{I_+},j\in{I_-}}\frac{1}{2}(\zeta^2-1)-\frac{1}{4}\sum_{i,j\in{I_+}}(a_i^2+a_j^2)^2-\frac{1}{4}\sum_{i,j\in{I_-}}(a_i^2+a_j^2)^2$.

\subsection{$J$-frame potential}
Since the additive constants has no physical significance, so we can disregard the additive constants.
The $J$-frame potential of a $\zeta-J$-frame $\textbf{\textit{F}}=\{f_i:i\in{I}\}$ in a Krein space $\textbf{\textit{K}}$ will be defined by
\begin{equation}
FP_J(\{f_i:i\in{I}\})=\sum_{i\in{I_+}}\sum_{j\in{I_+}}|[f_i,f_j]|^2+\sum_{i\in{I_-}}\sum_{j\in{I_-}}|[f_i,f_j]|^2
\end{equation}
\begin{example}
Let us consider the Krein space as in example (3.5) and let us consider the collection of vectors :\\
 $\{f_1=(1,0,-\frac{1}{\sqrt6}),f_2=(0,1,-\frac{1}{\sqrt6}),
f_3=(1,1,-\frac{\sqrt2}{\sqrt3}),f_4=(\frac{1}{\sqrt5},\frac{1}{\sqrt5},\frac{\sqrt3}{\sqrt5}),f_5=(1,1,\sqrt3)\}$.
Then it is a $J$-frame for the above Krein space. Here $\{f_1,f_2,f_3\}$ is the collection of positive vectors and $\{f_4,f_5\}$ is the set of negative vectors. Here $M_+=\{(x,y,z)\in{\mathbb{R}^3}:\frac{1}{\sqrt2}x+\frac{1}{\sqrt2}y+\sqrt{3}z=0\}$ and $M_-=\{(x,y,z)\in{\mathbb{R}^3}:\frac{x}{\frac{1}{2}}=\frac{y}{\frac{1}{2}}=\frac{z}{\frac{\sqrt3}{2}}\}$.
Now by calculation we have $G_{M_+}(x,y,z)=(\frac{7}{8}x-\frac{1}{8}y-\frac{\sqrt6}{8}z,-\frac{1}{8}x+\frac{7}{8}y-\frac{\sqrt6}{8}z,
-\frac{1}{\sqrt6}(\frac{3}{4}x+\frac{3}{4}y-\frac{\sqrt6}{4}z)),~\textmd{for all}~{(x,y,z)\in{\mathbb{R}^3}}$. So $\gamma(G_{M_+})=\frac{\sqrt6}{\sqrt7}$.
Similarly $G_{M_-}(x,y,z)=\frac{2\sqrt3}{5}(z+\frac{1}{\sqrt3}x+\frac{1}{\sqrt3}y)(\frac{1}{2},\frac{1}{2},\frac{\sqrt3}{2}),~\textmd{for all}~{(x,y,z)\in{\mathbb{R}^3}}$.
We have $\gamma(G_{M_-})=\frac{2}{\sqrt5}$. So here $\zeta=\frac{1}{\sqrt2}(\sqrt{\frac{\sqrt5+2}{2\sqrt5}}+\sqrt{\frac{\sqrt5-2}{2\sqrt5}})
+\frac{1}{\sqrt2}(\sqrt{\frac{\sqrt7+\sqrt6}{2\sqrt7}}+\sqrt{\frac{\sqrt7-\sqrt6}{2\sqrt7}})$.

 So according to our definition the $J$-frame force between the vectors $f_1$ and $f_3$ is $FF_J(f_1,f_3)=\frac{4}{3}(f_1-f_3)$. Similarly the $J$-frame force between the vectors $f_4$ and $f_5$ is $FF_J(f_4,f_5)=\frac{2}{\sqrt5}(f_4-f_5)$. Also the $J$-frame force between the vectors $f_2$ and $f_5$ is $FF_J(f_2,f_5)=2\{\frac{\sqrt7}{\sqrt6}.\sqrt5.\zeta+(1-\frac{1}{\sqrt2})\}(f_2-f_5)$. We will also calculate $J$-potential between some vectors. The $J$-potential between $f_1$ and $f_2$ is given by
$P_J(f_1,f_2)=\frac{-589}{36^2}$. The $J$-potential between $f_4$ and $f_5$ is $(\frac{1}{5}-\frac{26^2}{4.25^2})$. Similarly the $J$-potential between $f_2$ and $f_5$ is $\frac{1}{2}(\zeta^2-1)$.\\

Now we want to calculate the $J$-potential of the weakly normalized version of the above system. So let us consider the collection of vectors $\textbf{\textit{F}}=\{\tilde{f}_1=\frac{6}{5}(1,0,-\frac{1}{\sqrt6}),\tilde{f}_2=\frac{6}{5}(0,1,-\frac{1}{\sqrt6}),\tilde{f}_3=
\frac{3}{4}(1,1,-\frac{\sqrt2}{\sqrt3}),\tilde{f}_4=5(\frac{1}{\sqrt5},\frac{1}{\sqrt5},\frac{\sqrt3}{\sqrt5}),\tilde{f}_5=(1,1,\sqrt3)\}$. Then $FP_J(\textbf{\textit{F}})=(3+\frac{72}{25^2}+\frac{36}{25})+26$.\\

In the following section of our work we will show that $\frac{p^2}{m}+\frac{q^2}{n}=\frac{9}{2}+4$ is the minimum value that can be attained by a weakly normalized $\zeta-J$-frame in a Krein space, and also the minimum is attained only when the given normalized $\zeta-J$-frame is a tight $\zeta-J$-frame.
Since $(3+\frac{72}{25^2}+\frac{36}{25})+26~>~\frac{9}{2}+4$, so we can also say that $\textbf{\textit{F}}$ is not $\zeta-J$-tight.
\end{example}
In this section we will give a necessary and sufficient condition for $J$-frame potential for a Krein space $\textbf{\textit{K}}$ to be minimum.
\begin{lemma}
Let $\textbf{\textit{F}}=\{f_i\}_{i=1}^M$ be a weakly normalized $\zeta-J$-tight frame in the Krein space $\textbf{\textit{K}}$ of dimension $N$. Then the $J$-frame potential $FP_J$ of $\textbf{\textit{F}}$ is $\frac{p^2}{m}+\frac{q^2}{n}$.
\end{lemma}
\begin{proof}
Since $\textbf{\textit{F}}=\{f_i\}_{i=1}^M$ be a $J$-tight frame in the Krein space $\textbf{\textit{K}}$, so $\{f_i:i\in{I_+}\}\subset{S_{M_+}}$ is
a tight frame for the Hilbert space $(M_+,[~.~])$ with frame bound $A_+$. Also $\{f_i:i\in{I_-}\}$ is a tight frame for the Hilbert space $(M_-,-[~.~])$ with frame bound $A_-$. We have $|\textbf{\textit{F}}|=M$. Let $|I_+|=p$ and $|I_-|=q$. Also since $dim~\textbf{\textit{K}}=N$, let $dim~M_+=m$ and $dim~M_-=n$. Now
\begin{equation*}
\begin{split}
FP_J(\{f_i:i=1,\ldots,M\}) & =\sum_{i\in{I_+}}\sum_{j\in{I_+}}|[f_i,f_j]|^2+\sum_{i\in{I_-}}\sum_{j\in{I_-}}|[f_i,f_j]|^2\\
& =\sum_{i\in{\{1,\ldots,p\}}}A_+[f_i,f_i]+\sum_{j\in{\{1,\ldots,q\}}}A_-[f_i,f_i]\\
& =\frac{p^2}{m}+\frac{q^2}{n}.~~(\because~A_+=\frac{p}{m}~and~A_-=\frac{q}{n})\\
\end{split}
\end{equation*}
\end{proof}
\begin{theorem}
Let $\textbf{\textit{K}}$ be a Krein space of dimension $N$ and $\{f_i:i\in{I}\}$ be any weakly normalized $\zeta-J$-frame for $\textbf{\textit{K}}$. Also let $M_+=\overline{span}\{f_i:i\in{I_+\}}$ and $M_-=\overline{span}\{f_i:i\in{I_-\}}$. We assume that $dim{M_+}=m$, $dim{M_-}=n$, $|I_+|=p$ and $|I_-|=q$. Then the minimum value of the $J$-frame potential for a weakly normalized $\zeta-J$-frame is $\frac{p^2}{m}+\frac{q^2}{n}$ and this minimum is attained exactly when the frame vectors form a weakly normalized $\zeta-J$-tight frame for $\textbf{\textit{K}}$ for some $\zeta\in[\sqrt{2},2)$.
\end{theorem}
\begin{proof}
From the preceding lemma, we know that if $\textbf{\textit{F}}$ is a weakly normalized $\zeta-J$-tight frame then its associated frame potential is $\frac{p^2}{m}+\frac{q^2}{n}$. It remains to be shown that $\frac{p^2}{m}+\frac{q^2}{n}$ is a lower bound on the set of achievable $\zeta-J$-frame potentials among such collections $\textbf{\textit{F}}$, and that every collection attaining this lower bound is a $\zeta-J$-tight frame with the given $\zeta\in[\sqrt{2},2)$.

Let $\{f_i\}_{i=1}^M$ be a weakly normalized $\zeta-J$-frame for the Krein space $\textbf{\textit{K}}$. Then $\{f_i\}_{i\in{I_+}}$ is a frame for the Hilbert space $(M_+,[~.~])$. Similarly $\{f_i\}_{i\in{I_-}}$ is a frame for the Hilbert space $(M_-,-[~.~])$. Let $S_{G_+}$ is the frame operator for $\{f_i\}_{i\in{I_+}}$, where $G_+=G_{M_+}|_{M_+}\in{GL(M_+)}$. Let $\lambda_1,\ldots~,\lambda_m$ be the eigenvalues of $S_{G_+}$, counting multiplicity.

Then $\sum_{i\in{I_+}}\sum_{j\in{I_+}}|[f_i,f_j]|=\sum_{i\in{I_+}}\sum_{j\in{I_+}}|G_+(i,j)|^2=\|G_+\|_F^2=tr({G_+}^2)\\
=tr({S_{G_+}}^2)$, where $G_+$ is the Grammian operator of $\{f_i\}_{i\in{I_+}}$ in $(M_+,[~.~])$ and $\|.\|_F$ is the Frobenius norm.

Similarly let $S_{G_-}$ is the frame operator for $\{f_i\}_{i\in{I_-}}$, where $G_-=G_{M_-}|_{M_-}\in{GL(M_-)}$. Let $\mu_1,\ldots~,\mu_m$ be the eigenvalues of $S_{G_-}$, counting multiplicity. So $\sum_{i\in{I_-}}\sum_{j\in{I_-}}|[f_i,f_j]|=\sum_{i\in{I_-}}\sum_{j\in{I_-}}|G_-(i,j)|^2=\|G_-\|_F^2=tr({G_-}^2)=tr({S_{G_-}}^2)$, where $G_-$ is the Grammian operator of $\{f_i\}_{i\in{I_-}}$ in $(M_-,-[~.~])$ and $\|.\|_F$ is the Frobenius norm.

We have
\begin{equation*}
\begin{split}
		FP_J(\{f_i\}_{i=1}^M) & =tr({G_+}^2)+tr({G_-}^2)\\
		& =tr({S_{G_+}}^2)+tr({S_{G_-}}^2)\\
\end{split}
\end{equation*}

Now $tr(G_+)=tr(S_{G_+})=\sum_{i=1}^{p}[f_i,f_i]=p$. Similarly $tr(G_-)=tr(S_{G_-})=\sum_{i=1}^{q}-[f_i,f_i]=q$.
So, minimizing the frame potential under our constraint means minimizing $tr({G_+}^2)+tr({G_-}^2)=tr({S_{G_+}}^2)+tr({S_{G_-}}^2)=(\lambda_{1}^{2}+\ldots~+\lambda_{m}^{2})+(\mu_{1}^{2}+\ldots~+\mu_{n}^{2})$ under the constraint $tr(G_+)+tr(G_-)=tr(S_{G_+})+tr(S_{G_-})=(\lambda_{1}+\ldots~+\lambda_{m})+(\mu_{1}+\ldots~+\mu_{n})=p+q$, so we apply Lagrange multipliers theorem.
Here we have the following optimization problem :
\begin{equation}
\begin{split}		\textmd{Minimize~}f(\lambda_{1},\ldots~,\lambda_{m},~\mu_{1},\ldots~,\mu_{n})&=\lambda_{1}^{2}+\ldots~+\lambda_{m}^{2}+\mu_{1}^{2}+\ldots~+\mu_{n}^{2}.\\
\textmd{Subject~to~}g(\lambda_{1},\ldots~,\lambda_{m},~\mu_{1},\ldots~,\mu_{n})&=p+q.\\
\end{split}
\end{equation}
A careful observation of Lagrange multipliers theorem yields that to obtain the solution of $\nabla{f(\lambda_{1},\ldots~,\lambda_{m},~\mu_{1},\ldots~,\mu_{n})}
=\lambda{g(\lambda_{1},\ldots~,\lambda_{m},~\mu_{1},\ldots~,\mu_{n})}$ and $g(\lambda_{1},\ldots~,\lambda_{m},~\mu_{1},\ldots~,\mu_{n})=p+q$, we can split the optimization problem into the above two separated optimization problem \textit{i.e.} we will find the solution of $\nabla{f_1(\lambda_{1},\ldots~,\lambda_{m})}=\lambda{g_1(\lambda_{1},\ldots~,\lambda_{m})},~g_1(\lambda_{1},\ldots~,\lambda_{m})=p$ and $\nabla{f_2(\mu_{1},\ldots~,\mu_{n})}=\lambda{g_2(\mu_{1},\ldots~,\mu_{n})},~g_2(\mu_{1},\ldots~,\mu_{n})=q$. So to minimize the frame potential under our constraint means minimizing $tr({G_+}^2)=tr({S_{G_+}}^2)=\lambda_{1}^{2}+\ldots~+\lambda_{m}^{2}$ under the constraint $tr(G_+)=tr(S_{G_+})=\lambda_{1}+\ldots~+\lambda_{m}=p$, which implies that the minimizers satisfy $\lambda_{i}=\frac{p}{m}$ and minimizing $tr({G_-}^2)=tr({S_{G_-}}^2)=\mu_{1}^{2}+\ldots~+\mu_{n}^{2}$ under the constraint $tr(G_+)=tr(S_{G_+})=\mu_{1}+\ldots~+\mu_{n}=q$, we have that the minimizers satisfy $\mu_{i}=\frac{q}{n}$.

Therefore, the frame potential $FP_J$ has minimum value $\frac{p^2}{m}+\frac{q^2}{n}$ which is attained iff the eigenvalues are all equal in absolute value \textit{i.e.} $\lambda_i=\frac{p}{m}$ and $\mu_i=\frac{q}{n}$. So $S_{G_+}=\frac{p}{m}I_{M_+}$ and $S_{G_-}=\frac{q}{n}I_{M_-}$, hence $\textbf{\textit{F}}$ is a $J$-tight frame for $\textbf{\textit{K}}$.

Here the finite dimensional Krein space $\textbf{\textit{K}}$ is arbitrary but fixed after choice and so is the uniformly definite subspaces $M_+$ and $M_-$. So $\zeta=c_0(M_+,\verb"C")+c_0(M_-,\verb"C")$ is fixed throughout our proof.

Therefore, we have proved that the minimum of $J$-frame potential is attained exactly when the vectors form a weakly normalized $\zeta-J$-tight frame for $\textbf{\textit{K}}$.
\end{proof}

{\bf Acknowledgements.} Shibashis Karmakar, Sk. Monowar Hossein and Kallol Paul gratefully acknowledge the support of Jadavpur University, Kolkata and Aliah University, Kolkata for providing all the facilities when the manuscript was prepared. Shibashis Karmakar also acknowledges the financial support of CSIR, Govt. of India.


\end{document}